\documentclass[bibalpha]{amsart}
\usepackage{amssymb}
\usepackage{pb-diagram}

\newtheorem{theorem}{Theorem}[section]
\newtheorem{thm}[theorem]{Theorem}
\newtheorem{prop}[theorem]{Proposition}
\newtheorem{lem}[theorem]{Lemma}

\newtheorem{fact}[theorem]{Fact}
\newtheorem{cor}[theorem]{Corollary}

\newtheorem{corollary}[theorem]{Corollary}

\newtheorem{ques}[theorem]{Open Question}
\newtheorem*{thmA}{Theorem A}
\newtheorem*{factA}{Fact}
\newtheorem*{thmB}{Theorem B}
\newtheorem*{thmC}{Theorem C}
\newtheorem*{thmD}{Theorem D}
\newtheorem*{thmE}{Theorem E}
\newtheorem*{cors}{Corollary}

\theoremstyle{definition}
\newtheorem{defn}[theorem]{Definition}

\theoremstyle{remark}

\newcommand{\gr}{{\operatorname{gr}}}

\newcommand{\cl}{\operatorname{Cl}}
\newcommand{\bd}{\operatorname{Bd}}

\newcommand{\Int}{\operatorname{Int}}

\newcommand{\Cal}[1]{\ensuremath{\mathcal{#1}}}

\newcommand{\Z}{\mathbb{Z}}
\newcommand{\N}{\mathbb{N}}
\newcommand{\Q}{\mathbb{Q}}
\newcommand{\R}{\mathbb{R}}
\newcommand{\DSig}{D_{\Sigma}}
\newcommand{\FSig}{F_{\sigma}}

\begin{document}
\title{How to avoid a compact set}
\thanks{The second author was partially supported by NSF grants DMS-1300402 and DMS-1654725. The first and the third author were partially supported by the European Research Council under the European Union's Seventh Framework Programme (FP7/2007-2013) / ERC Grant agreement no.\ 291111/ MODAG}


\author{Antongiulio Fornasiero}
\address{Einstein Institute of Mathematics \\ The Hebrew University of Jerusalem\\ Edmond J. Safra Campus\\ Givat Ram. Jerusalem, 9190401, Israel}
\email{antongiulio.fornasiero@gmail.com}

\author{Philipp Hieronymi}
\address{Department of Mathematics\\University of Illinois at Urbana-Champaign\\1409 West Green Street\\Urbana, IL 61801}
\email{phierony@illinois.edu}
\urladdr{http://www.math.illinois.edu/\textasciitilde phierony}

\author{Erik Walsberg}
\email{erikw@illinois.edu}
\date{\today}

\maketitle

\begin{abstract}
A first-order expansion of the $\R$-vector space structure on $\R$ does not define every compact subset of every $\R^n$ if and only if topological and Hausdorff dimension coincide on all closed definable sets. Equivalently, if $A \subseteq \R^k$ is closed and the Hausdorff dimension of $A$ exceeds the topological dimension of $A$, then every compact subset of every $\R^n$ can be constructed from $A$ using finitely many boolean operations, cartesian products, and linear operations. The same statement fails when Hausdorff dimension is replaced by packing dimension.
\end{abstract}

\section{Introduction}

We study metric dimensions and general geometric tameness of definable sets in expansions of the ordered real additive group $(\R,<,+)$. This is an important part of the so called \emph{tameness program} first outlined by Miller \cite{Miller-tame}. As we hope our results are of interest to logicians and metric geometers alike,  we begin with an essentially logic free motivation for, and description of, the main results of this paper. We then describe some of the more technical results and their connections to notions from pure model theory for readers with background in logic.

\subsection*{For metric geometers} A \textbf{structure} on $\R$ is a sequence $\mathfrak{S} := (\mathfrak{S}_m)_{m=1}^{\infty}$ such that for each $m$:
\begin{enumerate}
\item $\mathfrak{S}_m$ is a boolean algebra of subsets of $\R^m$.
\item If $X \in \mathfrak{S}_m$ and $Y \in \mathfrak{S}_{n}$, then $X \times Y\in \mathfrak{S}_{m+n}$.
\item If $1\leq i \leq j \leq m$, then $\{ (x_1,\ldots,x_m) \in \R^m \ : \ x_i = x_j\} \in \mathfrak{S}_m$.
\item If $X\in \mathfrak{S}_{m+1}$, then the projection of $X$ onto the first $m$ coordinates is in $\mathfrak{S}_m$.
\end{enumerate}
Let $K$ be a subfield of $\R$.
We say a structure that satisfies the following three conditions is a \textbf{$K$-structure}.
\begin{enumerate}
\item[(5)] If $X \in \mathfrak{S}_m$ and $T : \R^m \to \R^n$ is $K$-linear, then $T(X)\in \mathfrak{S}_n$,
\item[(6)] If $1\leq i \leq j \leq m$, then $\{ (x_1,\ldots,x_m) \in \R^m\ : \ x_i < x_j\} \in \mathfrak{S}_m$.
\item[(7)] the singleton $\{ r\} \in \mathfrak{S}_1$ for every $r\in \R$.
\end{enumerate}
We say that $X\subseteq \R^n$ \textbf{belongs} to $\mathfrak{S}$ (or $\mathfrak{S}$ \textbf{contains} $X$)  if $X \in \mathfrak{S}_n$. We say $\mathfrak{S}$ \textbf{avoids} $X$ if $X \notin \mathfrak{S}_n$.\newline

Given $A \subseteq \R^n$, we let $\mathfrak{M}_K(A)$ be the smallest $K$-structure containing $A$. When we do not specify $K$, we mean that $K=\R$. In particular, $\mathfrak{M}(A)=\mathfrak{M}_{\R}(A)$.
Note that $\mathfrak{M}_K(\emptyset) \subseteq \mathfrak{M}_K(A)$ for any $A$.
It is easy to see that a solution set of a finite system of $K$-affine inequalities is in $\mathfrak{M}_K(\emptyset)$.
In particular every interval and every cartesian product of intervals is in $\mathfrak{M}_K(\emptyset)$.
Every polyhedron is in $\mathfrak{M}(\emptyset)$. \newline

\noindent It is natural to ask the following naive question:
\begin{center}
\emph{Given $A$, what can be said about $\mathfrak{M}(A)$?}
\end{center}
This is an important question in model theory in general and there are many classical results that provide answers to special cases. We give some examples.
\begin{itemize}
\item If $A$ is the solution set of a finite system of inequalities between affine functions then every element of $\mathfrak{M}(A)$ is a finite union of solution sets of affine inequalities.
This follows from the Fourier-Motzkin elimination for systems of linear inequalities.
This covers the case when $A$ is a polyhedron.
\item If $A$ is the solution set of a finite system of polynomial inequalities then every element of $\mathfrak{M}(A)$ is a finite union of solution sets of systems of polynomial inequalities.
This is a consequence of the Tarski-Seidenberg theorem, the real version of Chevalley's theorem on constructible sets.
This covers the case when $A$ is a zero set of a collection of polynomials.
\item Suppose that $A$ is a compact solution set of a finite system of inequalities between analytic functions.
Then every element of $\mathfrak{M}(A)$ is subanalytic.
In particular every element of $\mathfrak{M}(A)$ is a finite union of smooth submanifolds of Euclidean space.
This follows from Gabrielov's theorem of the complement, see \cite{subgab} for more information.
\end{itemize}
In the examples above $A$ is highly regular. However, regularity of $A$ does not guarantee regularity of $\mathfrak{M}(A)$. If $A$ is the graph of the sine function, or the set of integers, then $\mathfrak{M}(A)$ contains all closed sets by Hieronymi and Tychonievich \cite{HT}.\newline

\noindent In the present paper we ask the following:
\begin{center}
\emph{What can be said about $\mathfrak{M}(A)$ for a fractal $A$?}
\end{center}
Our main result is the following:

\begin{thmA}
Let $A \subseteq \R^n$ be closed and nonempty. If the Hausdorff dimension of $A$ exceeds the topological dimension of $A$, then $\mathfrak{M}(A)$ contains all compact subsets of all $\R^k$.
\end{thmA}
\noindent  As topological dimension (see Definition \ref{defn:topdim}) is an integer, Theorem A shows that if $A$ is closed and has fractional Hausdorff dimension, then $\mathfrak{M}(A)$ contains all compact sets.
According to Mandelbrot \cite{Mandelbrot} a compact set is a \textbf{fractal} if its Hausdorff dimension exceeds its topological dimension. Thus our result can be restated as:
\begin{center}
\emph{A single fractal is enough to generate every compact set}.
\end{center}
It can be shown that an $\R$-structure that contains all compact sets contains all bounded Borel sets. Even more is true: every projective\footnote{Projective in sense of descriptive set theory. See \cite[Chapter V]{kechris}} subset of $[0,1]^k$ is in such a structure. Such sets are very complicated, even innocent questions like whether all projective sets are Lebesgue measureable are independent of ZFC. It is worth pointing out that such a structure still does not need to contain all closed sets. By Pillay, Scowcroft and Steinhorn \cite[Theorem 2.1]{PSS} $\mathfrak{M}(A)$ avoids the graph of multiplication on $\R$ whenever $A\subseteq \R^n$ is bounded.\newline


We obtain a stronger result for compact totally disconnected subsets of $\R^n$.
We rely on a straightforward modification of a special case of a theorem of Friedman and Miller \cite[Theorem A]{FM-sparse}.

\begin{factA}
Let $A \subseteq \R^n$ be compact and totally disconnected such that $T(A^k)$ is nowhere dense for every linear map $T : \R^{kn} \to \R$.
Then $\mathfrak{M}(A)$ avoids a compact set.
\end{factA}

In fact, it follows from Theorem A that $\mathfrak{M}(A)$ avoids the classical middle-thirds Cantor set for such an $A$.
We obtain a necessary\emph{ and} sufficient condition on $A$ for $\mathfrak{M}(A)$ to avoid a compact set by combining the previous fact with Theorem A and some standard facts about Hausdorff dimension.


\begin{thmB} Let $A\subseteq \R^n$ be compact totally disconnected. Then $\mathfrak{M}(A)$ avoids a compact set if and only if the Hausdorff dimension of $A^k$ is zero for every $k$.
\end{thmB}
\noindent It follows that $\mathfrak{M}(A)$ avoids a compact set when $A$ is countable and compact.\newline

Theorem A is essentially optimal. Using Wei, Wen and Wen \cite{WWW}, we observe that there are compact subsets of $\R$ with positive packing dimension all of whose cartesian powers have Hausdorff dimension zero. Thus we can deduce from Theorem B that Theorem A fails when Hausdorff dimension is replaced by packing dimension. It follows that a set with non-integer packing dimension is not necessarily enough to generate all compact sets. We will see that the reason why Theorem A holds for Hausdorff dimension and fails for packing dimension goes back to the fact that the Marstrand projection theorem (see Fact \ref{marstrand} for a statement) holds for Hausdorff dimension and fails for packing dimension.  The failure of Theorem A for packing dimension is in contrast to known results for larger structures. If $\mathfrak{S}$ is a structure that contains both the graph of multiplication on $\R$ and a closed set $A\subseteq \R^n$ for which packing dimension exceeds topological dimension, then $\mathfrak S$ contains every closed set by Hieronymi and Miller \cite{HM}. Indeed in that setting, it is enough for the Assouad dimension of $A$ to exceed the topological dimension of $A$.\newline

Let $K$ be a countable subfield of $\R$. In general, Theorem A is no longer true when \emph{$\mathfrak{M}(A)$} is replaced by \emph{$\mathfrak{M}_{K}(A)$}. For example, when $K$ is the field of rationals and $A$ is the middle-thirds Cantor set, it follows from B\"uchi \cite{Buchi} that $\mathfrak{M}_{\Q}(A)$ avoids a compact set (see Boigelot, Rassart and Wolper \cite{BRW} for details). Indeed, the statement that $\mathfrak{M}_{\Q}(A)$ avoids a compact set, holds not only for the middle-thirds Cantor set, but for a large class of fractals, called automatic fractals by Adamczewski and Bell \cite{AB}. Roughly, a compact set $A \subseteq \R^n$ is \textbf{automatic} if there is $b\in \N_{\geq 2}$ such that the set of base $b$ representations of elements of $A$ is recognized by a B\"uchi automaton\footnote{A \textbf{B\"uchi automaton} is a finite state automaton that accepts infinite input sequences. For details, see for example Khoussainov and Nerode \cite{automata}}. Many famous fractals such as the middle-thirds Cantor set, the Sierpi\'nski carpet, and the Menger sponge are automatic. Metric properties of such fractals have recently been studied from a logical viewpoint by Charlier, Leroy, and Rigo \cite{CLR}. This paper continues their inquiry on the connection between metric geometry and definability, albeit with a different focus.\newline

 Readers familiar with first-order logic have surely already observed that the sets in $\mathfrak{M}_{K}(A)$ are exactly the definable sets in a certain expansion of the real ordered additive group. To make this precise, given a subfield $K$ of $\R$, we denote by $\mathcal{M}_K$ the expansion of $(\R,<,+)$ by all functions $\R \to \R$ of the form $x \mapsto \lambda x$ for $\lambda \in K$. A subset of $X \subseteq \R^n$ is in $\mathfrak{M}_K(A)$ if and only if $X$ is definable with parameters in $(\mathcal{M}_K,A)$. For various reasons model theorists prefer talking about expansions and definability rather than about structures. We follow this custom throughout the rest of the paper. Readers unfamiliar with these notions can take
\begin{itemize}
\item \emph{`an expansion of $\mathcal{M}_K$'} to mean \emph{`a structure containing all sets in $\mathfrak{M}_K$'},
\item  \emph{`definable in $(\mathcal{M}_K,A)$'} to mean \emph{`belongs to $\mathfrak{M}_K(A)$'}.
\end{itemize}
Note that $\mathcal{M}_{\mathbb{Q}} = (\R,<,+)$ and $\mathcal{M}_{\mathbb{Q}}(A) = (\R,<,+,A)$.
We refer the reader to Marker \cite{Marker} for an introduction to first-order logic.

\subsection*{For logicians}In proving Theorem A and B, we produce further results of interest to logicians. These results fit into the overall research program
of studying the consequences of model-theoretic tameness on the topology and geometry of definable sets in expansions of $(\R,<,+)$. We continue our investigation, begun in \cite{HW-Monadic}, of the connection between combinatorial tameness properties invented by Shelah  (neostability)  and geometric tameness properties of subsets of Euclidean space as considered by Miller \cite{Miller-tame} (tame geometry).\newline

We let $\Cal B$ be the two-sorted structure $(\mathcal{P}(\mathbb{N}), \mathbb{N},\in,+1)$, where $\mathcal P(\mathbb{N})$ is the power set of $\mathbb{N}$. In \cite{HW-Monadic} we initiated a study of structures that do not define an isomorphic copy of $\Cal B$. A structure that interprets $\Cal B$ is easily seen to violate all known Shelah-style combinatorial tameness properties such as $\operatorname{NIP}$ or $\operatorname{NTP2}$ (see e.g. Simon \cite{Simon-Book} for definitions). Thus the tameness condition \emph{`does not define an isomorphic copy of $\Cal B$'} can be seen as a substantial generalization of the two aforementioned tameness notions. Here we consider a further generalization of this notion. Let $\mathcal R$ be a first-order expansion of $(\R,<)$. A definable subset of $\R^n$ is \textbf{$\omega$-orderable} if it is either finite or admits a definable ordering with order type $\omega$. We say $\mathcal R$ defines a dense $\omega$-orderable set if it defines an $\omega$-orderable subset of $\R$ that is dense in some open interval. By \cite[Theorem A]{HW-Monadic} $\Cal R$ defines an isomorphic copy of $\Cal B$ whenever $\Cal R$ defines a dense $\omega$-orderable set. Thus the statement \emph{`does not define a dense $\omega$-orderable set'} is implied by \emph{`does not define an isomorphic copy of $\Cal B$'}. However, the converse is not true. There is a closed subset $C$ of $\R$ such that $(\R,<,+,\cdot,C)$ defines an isomorphic copy of $\Cal B$, but does not define a countable dense subset of $\R$ (see Friedman et al \cite{FKMS} or \cite{H-TameCantor}).\newline

At first glance the tameness notion of not defining a dense $\omega$-orderable set looks arbitrary. The relevance of this new condition arises from the observation that in expansions that define scalar multiplication by uncountably many numbers, the statement \emph{`does not define a dense $\omega$-orderable set'} is equivalent to the arguably weakest tameness property.

\begin{thmC} Let $K$ be an uncountable subfield of $\R$ and $\Cal R$ be an expansion of $\mathcal{M}_K$. Then the following are equivalent:
\begin{enumerate}
\item $\Cal R$ does not define a dense $\omega$-orderable set.
\item $\Cal R$ avoids a compact set.
\end{enumerate}
\end{thmC}
\noindent  We say that an expansion $\Cal R$ of $(\R,<)$ \textbf{avoids a compact set} if there is a compact subset of $\R^n$ not definable in $\Cal R$.
Observe that by Pillay, Scowcroft and Steinhorn \cite{PSS} we cannot replace $(2)$ with `\textit{$\Cal R$ avoids a closed set}'. In the case when $\Cal R$ expands the real field, Theorem C follows easily from \cite[Theorem 1.1]{discrete} and can therefore be recognized as a strong generalization of that result. An expansion of $(\R,<,+)$ that defines every compact set interprets second-order arithmetic. Thus the requirement of avoiding a compact set is as weak as a reasonable notion of model-theoretic tameness can be. The equivalence above fails when $K$ is countable. For example, let $\lambda$ be a quadratic irrational and $K = \mathbb{Q}(\lambda)$. Then $\mathbb{Z} + \lambda \mathbb{Z}$ is an $(\mathcal{M}_K,\Z)$-definable dense $\omega$-orderable set. By \cite{H-multiplication}, $(\mathcal{M}_K,\Z)$ avoids a compact set. It is worth pointing out that while $(\mathcal{M}_K,\Z)$ does not satisfy any Shelah-style combinatorial tameness property, its theory is still decidable and is in fact bi-interpretable with the theory of $\mathcal B$. Thus there are expansions of $(\R,<,+)$ that define dense $\omega$-orderable sets and are model-theoretically tame in some sense. Other examples of such structures are given in \cite{CLR} and \cite{BRW}. We regard this as evidence that this notion of tameness is worth studying in its own right.\newline

Let $\Cal R$ be an expansion of $(\R,<,+)$ that does not define a dense $\omega$-orderable set. Here we continue the study of tame topology in such structures. In the following \emph{`definable'} means definable in $\Cal R$, possibly with parameters. Most of our results concern $\DSig$ sets, a definable analogue of $\FSig$ sets. These sets were first studied by Miller and Speissegger \cite{MS99}. Here we use a slight modification of their original definition that first appeared in Dolich, Miller, and Steinhorn \cite{DMS1}.

\begin{defn} We say that $A\subseteq \R^n$ is $\boldsymbol{\DSig}$ if there is a definable family $\{ X_{r,s} : r,s > 0\}$ such that
 for all $r,r',s,s'> 0$:
\begin{itemize}
\item[(i)] $X_{r,s}$ is a compact subset of $\R^n$,
\item[(ii)] $X_{r,s} \subseteq X_{r',s'}$ if $r\leq r'$ and $s \geq s'$,
\item[(iii)] $A = \bigcup_{r,s} X_{r,s}$.
\end{itemize}
We say that the family $\{ X_{r,s} : r,s > 0\}$ witnesses that $A$ is $D_\Sigma$.
\end{defn}

Note in particular that closed definable sets are $\DSig$.
It is easy to see that the image of a $\DSig$ set under a continuous definable function is $\DSig$.
 It is not difficult to show that open definable sets and boolean combinations of open definable sets are $\DSig$ (see \cite[1.10]{DMS1}). Every $\DSig$ set is $\FSig$.
The following Theorem is crucial.

\begin{thmD} Let $A \subseteq \R^n$ be $\DSig$. Then the following are equivalent:
\begin{itemize}
\item $A$ does not have interior,
\item $A$ is nowhere dense,
\item $A$ has $n$-dimensional Lebesgue measure zero.
\end{itemize}
\end{thmD}

Suppose $\{ X_{r,s} : r,s > 0\}$ is a definable family of compact sets that witnesses that $A$ is $D_\Sigma$.
It follows from the Baire Category Theorem that if $A$ has interior then some $X_{r,s}$ has interior.
Theorem D implies a strong definable form of the Baire Category Theorem:
\begin{cors}
Suppose that $\{ X_{r,s} : r,s> 0 \}$ is a definable family of closed subsets of $\R^n$ such that $X_{r,s} \subseteq X_{r',s'}$ when $r \leq r'$ and $s \geq s'$.
If the union of the $X_{r,s}$ is somewhere dense, then there are $r',s' > 0$ such that $X_{r',s'}$ has interior.
\end{cors}
Our proof of Theorem D closely follows an argument from \cite{MS99}. As a consequence we show that topological dimension is well-behaved on $\DSig$ sets.
Let $\dim_T(A)$ be the topological dimension of $A$.

\begin{thmE} Let $A\subseteq \R^m$ and $B\subseteq \R^n$ be $\DSig$ sets. Let $f : \R^m \to \R^n$ be a continuous definable function. Then
\begin{enumerate}
\item $\dim_T(A)=\dim_T(\cl(A))$ and $\dim_T \bd(A) < m$.
\item $\dim_T A\geq \dim_T f(A)$,
\item if there is a $d\in \N$ such that $\dim_T f^{-1}(x) = d$ for all $x\in B$, then
\[
\dim_T f^{-1}(B) = \dim_T B + d.
\]
\end{enumerate}
\end{thmE}

By (2) above there is no continuous definable surjection $\R^m \to \R^{n}$ when $m < n$.
By (3) above we have $\dim_T(A \times B) = \dim_T(A) + \dim_T(B)$ for $D_\Sigma$ sets $A,B$.
This does not hold for all compact subsets of Euclidean space.
Pontryagin \cite{pont} constructed compact subsets $A,B \subseteq \R^4$ such that $\dim_T(A) = \dim_T(B) = 2$ and $\dim_T(A \times B) = 3$.
Theorem E is new even for $\operatorname{NIP}$ expansions of $(\R,<,+)$.\newline

Theorem E fails when the assumption that $A$ and $B$ are $\DSig$ is dropped. Consider the expansion $(\R,<,+,\Q)$. Using \cite[5.8]{DMS1} it is not hard to check that this structure does not define a dense $\omega$-orderable set. However, the definable set $[ (\R \setminus \mathbb{Q}) \times \{ 1 \} ] \cup [ \mathbb{Q} \times \{0\}]$ has topological dimension zero, while its projection onto the first coordinate has topological dimension one.\newline

We now explain how to derive Theorem A from Theorem D and E. Let $\dim_H(A)$ be the Hausdorff dimension of $A$. Let $A\subseteq \R^n$ be closed and definable such that $\dim_T (A) < \dim_H(A)$. After replacing $A$ with $A^k$ for sufficiently large $k$, we may suppose $\dim_T (A) +1< \dim_H(A)$. Applying the Marstrand projection theorem we obtain an orthogonal projection $\rho :\R^n \to \R^{\dim_T(A)+1}$ such that $\rho(A)$ has positive $(\dim_T(A)+1)$-dimensional Lebesgue measure. This projection is definable if $\Cal R$ expands $\Cal M_{\R}$. Thus $\rho(A)$ is $\DSig$. By Theorem D, $\rho(A)$ has interior. By Theorem E, $\dim_T(A) \geq \dim_T \rho(A) = \dim_T(A)+1$, contradiction.\newline

A few words of attribution are in order.  The underlying idea behind the above argument goes back to Edgar and Miller \cite{EM01}. Building on their work, Fornasiero in the unpublished manuscript \cite{F-Hausdorff} devises the above strategy to prove Theorem A for expansions of the real field. The manuscript \cite{F-Hausdorff} also contains a different proof of Theorem D and Theorem E for expansions of the real field.\newline

Here is an overview of the paper. In Section 2 we explain the necessary prerequisites from descriptive set theory and metric geometry. In Section 3 we recall some results from \cite{HW-Monadic} about expansions that do not define dense $\omega$-orderable sets, and prove Theorem C. The proof of Theorem D is in Section 4. In Section 5 we prove statement (2) of Theorem E and deduce Theorem A. Section 6 contains a proof of the other statements of Theorem E. Finally we establish Theorem B in Section 7 and explain why Theorem A fails for packing dimension.

\subsection*{Acknowledgements} The authors thank Chris Miller for envisioning this research area and for supporting our work in it. They also thank Tapio Rajala for answering some questions, and William Balderrama for doing a reading course with the second author.
They thank Samantha Xu for useful feedback and thank logic groups in Bogota, Chicago, Paris, New York, and Urbana for listening to talks on this subject and giving useful feedback.

\section{Prerequisites}

\subsection*{Notation} Let $A \subseteq \R^n$. We denote by $\cl (A)$ the closure of $A$, by $\Int(A)$ the interior of $A$, and by $\bd (A)$ the boundary $\cl(A) \setminus \Int(A)$ of $A$. Whenever $A \subseteq \R^{m+n}$ and $x\in \R^m$, then $A_x$ denotes the set $\{ y \in \R^n \ : \ (x,y) \in A\}$. We always use $i,k,l,m,n$ for natural numbers and $r,s,t,\lambda,\epsilon,\delta$ for real numbers.
Given a function $f : A \to B$ we let $\gr(f) \subseteq A \times B$ be the graph of $f$.
Given $x = (x_1,\ldots,x_n) \in \R^n$ we let
$$ \|x\| = \max \{ |x_1|, \ldots, |x_n| \}$$
be the $l_{\infty}$ norm of $x$.
We use the $l_\infty$ norm as $\|x\|$ is, unlike the $l_2$ norm, an $(\R,<,+)$-definable function of $x$.

\subsection*{Descriptive set theory} For basic notions and results from descriptive set theory, we refer the reader to Kechris \cite{kechris}. The fact that an $F_\sigma$ set has interior if and only if it is not meager, yields the following strengthening of the Kuratowski-Ulam theorem for $F_\sigma$ sets which we use repeatedly (compare \cite[1.5(3)]{MS99}).

\begin{fact}\label{KU}
Let $A \subseteq \mathbb{R}^{m + n}$ be $F_\sigma$. Then $A$ has interior if and only if the set
$\{ x \in \R^m : A_x \text{ has interior}\}$ is non-meager.
\end{fact}

\subsection*{Metric Dimensions}
Let $X\subseteq \R^n$. We denote the \textbf{Hausdorff dimension} of $X$ by $\dim_H X$ and the \textbf{packing dimension} of $X$ by $\dim_P(X)$. We refer the reader to Falconer \cite{Falconer} for the precise definitions of these dimensions. For the reader's convenience we recall the properties of Hausdorff and packing dimension used in this paper.

\begin{fact} \label{Hsadditive}
Let $A,B \subseteq \R^n$. Then
\begin{enumerate}
\item If $A$ has positive $n$-dimensional Lebesgue measure then $\dim_H A = n$.
\item $\dim_H A \leq \dim_H B$ if $A\subseteq B$.
\item $\dim_H T(A) \leq \dim_H A$ for every linear map $T : \R^n \to \R^m$.
\item $ \dim_H A + \dim_HB \leq \dim_H(A \times B).$
\item $ \dim_H A = \sup_{i \in \N} \dim_H A_i$ whenever $A=\bigcup_{i\in \N} A_i$.
\end{enumerate}
Items (1) - (3) and (5) of Fact \ref{Hsadditive} hold with $\dim_P$ in place of $\dim_H$. Moreover, $\dim_P(A \times B) \leq dim_P A + \dim_P B$ and $\dim_H A \leq \dim_P A$ for every $A, B \subseteq \R^n.$
\end{fact}

\noindent The deepest fact that we use about Hausdorff dimension is a special case of the Marstrand projection theorem (see \cite[Theorem 6.2]{Falconer}).

\begin{fact}[Marstrand projection theorem]\label{marstrand}
Let $B \subseteq \mathbb{R}^n$ be Borel such that $\dim_H B > m$.
Then there is an orthogonal projection $\rho : \mathbb{R}^n \to \mathbb{R}^m$ such that $\rho(B)$ has positive $m$-dimensional Lebesgue measure.
\end{fact}

\noindent For our purposes, the main difference between packing and Hausdorff dimension is the failure of Fact~\ref{marstrand} for packing dimension. By J{\"a}rvenp{\"a}{\"a} \cite{Jarvenpaa} there are subsets of $\R^n$ with packing dimension greater than $m$ such that all projections onto all $m$-dimensional subspaces are Lebesgue null.

\subsection*{Topological dimension} We now define the topological dimension of $X$. We refer to Engelking \cite{Engelking} for proofs. There are several different notions of dimension for topological spaces all of which fortunately agree on subsets of Euclidean space.

\begin{defn} \label{defn:topdim}
The \textbf{topological dimension} of $X$ is defined inductively: $\dim_T(\emptyset) = -\infty$ and if $X \neq \emptyset$ then $\dim_T(X)$ is the minimal $k$ such that for every $x \in X$ and neighbourhood $U$ of $x$ there is a neighbourhood $V$ of $x$ such that the closure of $V$ is contained in $U$ and $\dim_T(\bd(V) \cap X) < k$. We denote the topological dimension of $X$ by $\dim_T(X)$.
\end{defn}

\noindent This dimension is also called the small inductive dimension. We will need the following facts.

\begin{fact}\label{topdimmono}
Let $A,B \subseteq \R^n$. Then
\begin{enumerate}
\item $\dim_T A =n$ if and only if $A$ has nonempty interior in $\R^n$.
\item $\dim_T A \leq \dim_T B$ if $A\subseteq B$.
\item $ \dim_T A = \sup_{i \in \N} \dim_T (A_i)$ whenever $A=\bigcup_{i\in \N} A_i$ and each $A_i$ is closed.
\item if $A$ is $\FSig$ and $\dim_T A \geq m$ then there is a coordinate projection $\pi : \R^n \to \R^m$ such that $\dim_T\pi(A) =m$.
\item $\dim_T A \leq \dim_H A$.
\item $\dim_T A = 0$ if and only if $A$ is totally disconnected.
\end{enumerate}
\end{fact}
\noindent The compact case of $(4)$ is due to N\"obeling, see \cite[1.10.23]{Engelking}.
We view it as an analogue of the Marstrand projection theorem for topological dimension.
Note that the $\FSig$ case of $(5)$ follows by applying (in order) Fact~\ref{topdimmono}$(4)$, Fact~\ref{topdimmono}$(1)$, Fact~\ref{Hsadditive}$(1)$, and Fact~\ref{Hsadditive}$(3)$.

\subsection*{Naive dimension} We have to introduce one last notion of dimension that is frequently used in model theory. Let $X \subseteq \R^n$. The \textbf{naive dimension} of $X$, written $\dim X$, is $-\infty$ if $X$ is empty, and is otherwise the maximal $k$ for which there is a coordinate projection $\pi : \R^n \to \R^k$ such that $\pi(X)$ has interior.
It is easy to see that the following facts are true:

\begin{fact}\label{fact:naivedim} Let $A,B\subseteq \R^n$. Then
\begin{enumerate}
\item $\dim A=n$ if and only if $A$ has interior.
\item $\dim A \leq \dim B$ if $A\subseteq B$.
\item $\dim A \times B = \dim A + \dim B$.
\item $\dim A \leq \dim_H A$.
\item $\dim_T A \leq \dim A$ if $A$ is $\FSig$.
\end{enumerate}
\end{fact}
\noindent Note that Fact~\ref{fact:naivedim}$(5)$ follows from Fact~\ref{topdimmono}$(4)$ and Fact~\ref{topdimmono}$(1)$. While naive dimension is not studied in metric geometry, it has various desirable definability-theoretic properties. The most important is the following.

\begin{fact} Let $A\subseteq \R^{m + n}$ and $\ast$ in $\{\leq,\geq,<,>=\}$. Then
\[
\{ x \in \R^m \ : \ \dim A_x \ast k \}
\]
is definable in $(\R,<,A)$ for any $k$.
\end{fact}

\subsection*{$\DSig$ sets} We collect a few results about $\DSig$ sets in expansions of $(\R,<,+)$. Throughout this subsection $\Cal R$ is an expansion of $(\R,<,+)$.



\begin{fact}\label{basicDsigma}

\begin{enumerate}
\item The image of a $\DSig$ set under a continuous definable map is $\DSig$.
\item Finite intersections, and finite unions of $\DSig$ sets are $\DSig$.
\item If $A \subseteq \R^{m + n}$ is $\DSig$ then $A_x$ is $\DSig$ for every $x \in \R^m$.
\item Every definable set that is a boolean combination of closed sets is $\DSig$.
\item If $\{ A_t : t > 0 \}$ is a definable family of $D_\Sigma$ sets such that $A_s \subseteq A_t$ whenever $t < s$, then the union of the $A_t$ is $D_\Sigma$.
\end{enumerate}
\end{fact}
\begin{proof}
For (1), observe that a continuous map always maps compact sets to compact sets. For (2) and (4), see \cite[1.9(1)]{DMS1} and \cite[1.10(2)]{DMS1}, (3) follows from (2).
\end{proof}

\noindent The complement of a $\DSig$ set need not be $\DSig$ (see \cite[p.1378]{DMS1} for an example).

\begin{fact}\label{fact:dimdsig} Let $A \subseteq \R^{m + n}$ be $\DSig$. Then
\begin{enumerate}
\item $\{ x \in \R^m \ : \ A_x \hbox{ has interior }\}$ is $\DSig$,
\item $\{ x \in \R^m \ : \ \dim A_x \geq k \}$ is $\DSig$ for any $k$.
\end{enumerate}
\end{fact}
\begin{proof}
It is not hard to see that (2) follows from (1) and the fact that a projection of a $\DSig$ set is $\DSig$. In order to establish (1), let $\{ X_{r,s} : r,s > 0\}$ be a definable family that witnesses that $A$ is $\DSig$. Then
\begin{align*}
\{ x \in \R^m \ & : \ A_x \hbox{ has interior }\}\\
&=\bigcup_{r,s \in \R_{>0}} \left\{ x \in \R^m \ : \ \exists y_1,\ldots,y_{n} \in \R \quad \prod_{i = 1}^{n} [y_i,y_i+s] \subseteq X_{r,s,x}\right\}.
\end{align*}
One can check that since $X_{r,s}$ is compact for each $r,s$, so is
\[
\left\{ x \in \R^m \ : \ \exists y_1,\ldots,y_n \in \R \quad \prod_{i = 1}^{n} [y_i,y_i+s] \subseteq X_{r,s,x}\right\}.
\]
Now (1) follows.
\end{proof}

\section{Dense $\omega$-orderable sets}

In this section we establish basic facts about expansions of $(\R,<,+)$ that do not define dense $\omega$-orderable sets and prove Theorem C. Unless we say otherwise, $\Cal R$ is an expansion of $(\R,<,+)$ and \emph{`definable'} means \emph{`definable in $\Cal R$'} (possibly with parameters). We recall the following definition.

\begin{defn} A \textbf{dense $\omega$-orderable set} is a definable subset of $\R$ that is dense in some open interval and admits a definable ordering with order type $\omega$.
\end{defn}

\begin{fact}\label{fact:basicomega}
Let $D \subseteq \R^l$ be $\omega$-orderable, let $l,m,n\in \N$ and let $f: \R^{l} \to \R^m$ be definable. Then
 $D^n$ and $f(D)$ are $\omega$-orderable.
\end{fact}
\begin{proof}
The fact is clear if $D$ is finite, we suppose that $D$ is infinite.
Let $\prec$ be a definable ordering on $D$ with order type $\omega$. Let $\prec_{lex}$ be the lexicographic order on $D^n$ and let $\max_{\prec} : D^n \to D$ map $(x_1,\dots,x_n)$ to the $\prec$-maximum of $\{x_1,\dots x_n\}$. Note that both $\prec_{lex}$ and $\max_{\prec}$ are definable. Let $x=(x_1,\dots,x_n),y=(y_1,\dots,y_n) \in D^n$. We say that $x \prec_n y$ if either
\begin{itemize}
\item $\max_{\prec} x \prec \max_{\prec} y$ or
\item $\max_{\prec} x = \max_{\prec} y$ and $x\prec_{lex} y$.
\end{itemize}
Observe that $\prec_n$ is definable and  $(D^n,\prec_n)$ has order type $\omega$.\newline
Let $f : D \to \R^m$ be definable, and suppose that $f(D)$ is infinite. The $\omega$-order $\prec_f$ on $f(D)$ is given as follows: given $x,y \in f(D)$ we declare $x\prec_f y$ if there is a $d \in D$ such that $f(d)=x$ and $f(e)\neq y$ for all $e \prec d$.
\end{proof}

\noindent By \cite[Theorem A]{HW-Monadic}, an expansion of $(\R,<,+)$ that does not define the two-sorted structure $\Cal B:= (\Cal P(\N),\N,\in,+1)$ cannot define a dense $\omega$-orderable set. Such structures were studied in detail in \cite{HW-Monadic}. However, while the results in \cite{HW-Monadic} were stated for expansions that do not define $\Cal B$, many of their proofs only made use of the fact that such structures do not define dense $\omega$-orderable sets. In hindsight, this should been made clear in \cite{HW-Monadic}, but the authors did not anticipate the relevance of the weaker assumption. For the reader's convenience we recall the main results from the aforementioned paper that we use here.

\begin{fact}[{Weak Monotonicity Theorem \cite[Theorem E]{HW-Monadic}}] \label{fact:monotonicity}
Suppose that $\mathcal R$ does not define a dense $\omega$-orderable set.
Let $f: \R \to \R$ be a definable continuous function. Then there is a definable open dense $U \subseteq \R$ such that $f$ is strictly increasing, strictly decreasing, or constant on each connected component of $U$.
\end{fact}

\begin{fact}[{\cite[Lemma 3.4]{HW-Monadic}}] \label{fact:oldunary}
Suppose that $\mathcal R$ does not define a dense $\omega$-orderable set. Let $\{ X_{r,s} : r, s > 0 \}$ be a definable family of subsets of $\R$ such that
 for all $r,r',s,s' > 0$:
\begin{itemize}
\item[(i)] $X_{r,s}$ is nowhere dense,
\item[(ii)] $X_{r,s} \subseteq X_{r',s'}$ if $r\leq r'$ and $s \geq s'$.
\end{itemize}
Then $\bigcup_{r,s} X_{r,s}$ is nowhere dense.
\end{fact}

\begin{fact}[{\cite[Lemma 4.5]{HW-Monadic}}] \label{fact:oldunarylebesgue} Suppose that $\mathcal R$ does not define a dense $\omega$-orderable set.
Then a nowhere dense definable subset of $\R$ is Lebesgue null.
\end{fact}

\noindent Fact~\ref{fact:oldunary} and Fact~\ref{fact:oldunarylebesgue} yield Theorem D in the case $n=1$:

\begin{cor}\label{oldunary}
Suppose that $\mathcal R$ does not define a dense $\omega$-orderable set and that $A$ is a $\DSig$ subset of $\R$.
Then the following are equivalent:
\begin{enumerate}
\item $A$ does not have interior.
\item $A$ is nowhere dense.
\item $A$ is Lebesgue null.
\end{enumerate}
\end{cor}
\begin{proof}
It is clear that $(3)$ implies $(1)$.
Fact~\ref{fact:oldunarylebesgue} shows that $(2)$ implies $(3)$.
We show that $(1)$ implies $(2)$.
Let $\{ X_{r,s} : r,s > 0 \}$ be a definable family that witnesses that $A$ is $\DSig$. Suppose $A$ does not have interior. It follows that each $X_{r,s}$ does not have interior. As each $X_{r,s}$ is compact, $X_{r,s}$ is nowhere dense for all $r,s > 0$. Fact \ref{fact:oldunary} implies $A$ is nowhere dense.
\end{proof}

\subsection*{Proof of Theorem C} In this subsection we give a proof of Theorem C. This proof is an application of the following theorem from \cite{HT}.

\begin{fact}[{\cite[Theorem A]{HT}}]\label{thm:prhie}
Let $\Cal R$ be an expansion of $(\R,<)$. If $\Cal R$ defines an open interval $U\subseteq\R$, a closed and discrete set $D\subseteq\R_{\geq 0}$, and functions $f: D\rightarrow\R$ and $g\colon\R^3\times D\rightarrow D$ such that:
\begin{enumerate}
\item[(i)] $f(D)$ is dense in $U$,
\item[(ii)] for every $a,b\in U$ and $d,e\in D$ with $a<b$ and $e\leq d$,
\[ \{c\in\R:g(c,a,b,d)=e\}\cap (a,b)\mbox{ has nonempty interior,}\]\end{enumerate}
then $\Cal R$ defines every subset of $D^n$ and every open subset of $U^n$ for every $n\in\N$.
\end{fact}

\noindent As $D$ is a closed and discrete subset of $\R_{\geq 0}$ the usual order on $\R$ induces an $\omega$-order on $D$.
Then $f(D)$ is $\omega$-orderable by Fact~\ref{fact:basicomega}.
  Indeed, a close inspection of the proof of \cite[Theorem A]{HT} reveals that the only property of $f$ and $D$ used is the fact that $f(D)$ is dense in $U$ and $\omega$-orderable.

\begin{prop}\label{prop:htplus}
Let $\Cal R$ be an expansion of $(\R,<)$. If $\Cal R$ defines an ordered set $(D,\prec)$, an open interval $U\subseteq\R$, and a function $g\colon\R^3\times D\rightarrow D$ such that
\begin{itemize}
\item[(i)] $(D,\prec)$ has order type $\omega$ and $D$ is dense in $U$,
\item[(ii)] for every $a,b\in U$ and $e,d\in D$ with $a<b$ and $e\preceq d$,
\[\{c\in\R:g(c,a,b,d)=e\}\cap(a,b)\mbox{ has nonempty interior,}\]
\end{itemize}
then $\Cal R$ defines every subset of $D^n$ and every open subset of $U^n$ for every $n\in \N$.
\end{prop}
\begin{proof}
Essentially the same proof as that of \cite[Theorem A]{HT} may be used to show the statement of Proposition~\ref{prop:htplus}. We leave these routine modifications to the reader.
\end{proof}

\noindent We now prove Theorem C. In fact, we prove a stronger result. For $D\subseteq \R$, we denote by $\Q(D)$ the subfield of $\R$ generated by $D$.

\begin{thm}
Let $\mathcal R$ be an expansion of $(\R,<,+)$.
\begin{enumerate}
\item If $\mathcal R$ defines every compact set, then $\mathcal R$ defines a dense $\omega$-orderable set.
\item If $\mathcal R$ defines a dense $\omega$-orderable set $D$ and there is $\lambda \in \R\setminus \Q(D)$ such that $\Cal R$ defines multiplication by $\lambda$, then $\mathcal R$ defines every compact set.
\end{enumerate}
\end{thm}

\noindent This implies Theorem C as the subfield generated by an $\omega$-orderable set is countable.

\begin{proof}
Suppose $\Cal R$ defines every compact set.
Let $D = \{ \frac{1}{n} : n \in \mathbb{N}, n \geq 1 \}$.
Let $f : D \to \mathbb{Q} \cap [0,1]$ be a bijection.
Let $G \subseteq \R^2$ be the graph of $f$ and let $C$ be the closure of $G$.
Then $C$ is compact as $G$ is bounded and it is easy to see that
$$ G = C \cap \{ (x,y) \in \R : x > 0 \}. $$
It follows that $G$ is definable, and hence so is $\Q\cap [0,1]$.
We define an $\omega$-order on $\mathbb{Q} \cap [0,1]$ by declaring $p \prec q$ if $f^{-1}(q) < f^{-1}(p)$.
\newline

\noindent Suppose that $\Cal R$ defines a dense $\omega$-orderable set $(D,\prec)$ and defines multiplication by a positive $\lambda \in \R\setminus \Q(D)$. First observe that there is a definable order $\prec'$ on $\lambda D$ such that multiplication by $\lambda$ is an order isomorphism between $(D,\prec)$ and $(\lambda D, \prec')$. Let $U\subseteq \R$ be an open interval in which $D$ is dense. As $\Cal R$ expands $(\R,<,+)$, all compact sets are definable if there is a definable function $g: \R^3 \times D \to D$ that satisfies condition (ii) of Proposition \ref{prop:htplus}.
We now follow the proof of Theorem C in \cite{HT}. Define a function $h_1: \R_{>0} \times D \times D \to \lambda D$ by declaring $h_1(u,d,e)$ to be the $\prec'$-minimal $x\in \lambda D$ such that
\begin{align*}
x\in \big(e,e+u\big) \hbox{ and } D_{\preceq d} \cap \big(e,x\big] = \emptyset.
\end{align*}
Let $h_2 : \R_{>0} \times D \times D \to \R$ be the function given by
\[
h_2(u,d,e) := h_1(u,d,e)-e.
\]
For fixed $d \in D$ and $u \in \R$, it follows from $\lambda \notin \Q(D)$ that the function $e \mapsto h_2(u,d,e)$ is injective on $D_{\preceq d}$.
Then let $g : \R^3 \times D \to D$ be defined such that if $b>a$, $g(c,a,b,d)$ is the $e \in D_{\preceq d}$ such that
$|(c-a)- h_2((b-a),d,e)|$ is minimal, and $g(c,a,b,d)=1$, if $b\leq a$.  It can be deduced from the injectivity of $h_2(u,d,-)$ that the function $g$ satisfies condition (ii) of Proposition \ref{prop:htplus}.
\end{proof}

\noindent As pointed out in the introduction, it is known that the expansion of $(\R,<,+)$ by the classical middle-thirds Cantor set avoids a compact set.
Thus the corollary below implies that this structure defines multiplication by $\lambda \in \R$ if and only if $\lambda$ is rational.

\begin{corollary}
Let $C \subseteq \R$ be the classical middle thirds Cantor set.
Let $\lambda \in \R$ be irrational.
Then $(\R,<,+, C, x \mapsto \lambda x)$ defines all compact sets.
\end{corollary}

\begin{proof}
By Theorem C it is enough to show that $(\R,<,+,C)$ defines a dense $\omega$-orderable set of rational numbers.
The complement of $C$ decomposes as a countable disjoint union of open intervals, the complementary intervals of $C$.
Let $E$ be the set of right endpoints of complementary intervals of $C$.
It follows from the construction of $C$ that $E \subseteq \mathbb{Q}$.
For each $e \in E$, let $\delta(e) \in \R \cup \{\infty\}$ be the length of the complementary interval with right endpoint $e$.
We put an ordering $\prec$ on $E$ by declaring $e \prec d$ if $\delta(e) > \delta(d)$ or if $\delta(e) = \delta(d)$ and $e < d$.
It was shown in \cite[Proof of Theorem B]{HW-Monadic} that such an ordering has order type $\omega$.
As $E$ is $\omega$-orderable it follows from Fact~\ref{fact:basicomega} that  $E - E := \{ e - e' : e, e' \in E \}$ is $\omega$-orderable.
It is well-known that $C - C=[-1,1]$. As $E$ is dense in $C$ we see that $E - E$ is dense in $[-1,1]$.
\end{proof}

\section{The Strong Baire Category Theorem}

In this section $\Cal R$ is an expansion of $(\R,<,+)$ that does not define a dense $\omega$-orderable set. We prove a key result, the strong Baire category theorem (or \textbf{SBCT} for short).
Our proof closely follows an argument of \cite{MS99}.

\begin{thm}\label{SBCT}
Let $A \subseteq \R^{n}$ be $\DSig$. Then $A$ either has interior or is nowhere dense.
Moreover, if $W \subseteq \R^{n}$ is open and $A$ is dense in $W$, then the interior of $A$ is dense in $W$.
\end{thm}

\begin{proof}
We proceed by induction on $n$. Observe that for each $n$, the second statement follows from the first statement, as $\cl(A)= \Int(A) \cup [\cl(A) \setminus \Int (A)]$.  The base case $n = 1$ follows from Fact~\ref{oldunary}. We suppose that $n \geq 2$ and that $A\cap U$ is dense in $U$ for some open set $U$. We now show that $A \cap U$ has interior. After shrinking $U$ if necessary we suppose that $U = V \times I$ where $V \subseteq \R^{n - 1}$ is definable and open and $I \subseteq \R$ is an open interval in $\R$. Observe that $A \cap U$ is $D_\Sigma$. Therefore, after replacing $A$ with $A \cap U$, we may assume that $A \subseteq U$.  Let $\pi : \R^n \to \R^{n - 1}$ be the projection onto the first $n - 1$ coordinates.
We show that
$$ B := \{ x \in V : A_x \text{ is dense in } I \} $$
is a comeager subset of $V$.
For this purpose it is enough to fix an open subinterval $J \subseteq I$ with rational endpoints and show that
$$ C : = \{ x \in V : A_x \cap J \neq \emptyset \} $$
contains an open dense subset of $V$.
Note that $C = \pi\big( A \cap (V \times J) \big)$. In particular, $C$ is $\DSig$.
As $A \cap (V \times J)$ is dense in $V \times J$, $C$ is dense in $V$.
The inductive assumption implies that $C$ contains an open dense subset of $V$.
Thus $B$ is comeager in $U$.\newline
We finish the proof that $A$ has non-empty interior. Each $A_x$ is $D_\Sigma$, so $A_x$ contains a dense open subset of $I$ whenever $x\in B$.
Thus $A_x$ has interior for comeager many $x \in V$. By Fact~\ref{KU}, $A$ has interior.
\end{proof}

\noindent We now prove the Lebesgue measure version of strong Baire category.

\begin{prop}\label{lmeasure}
Let $A\subseteq \R^n$ be $\DSig$.
Then $A$ either has interior or is Lebesgue null.
\end{prop}

\begin{proof}
Suppose that $A$ has positive $n$-dimensional Lebesgue measure.
We apply induction to $n$.
The base case $n = 1$ follows by Fact~\ref{oldunary}.
Suppose $n \geq 2$.
Let $\pi : \mathbb{R}^n \to \mathbb{R}^{n - 1}$ be the projection onto the first $n - 1$ coordinates.
Let $B$ be the set of $x \in \pi(A)$ such that $A_x$ has positive one-dimensional Lebesgue measure.
By Fubini's theorem, $B$ has positive $(n - 1)$-dimensional Lebesgue measure.
It follows from the base case that $A_x$ has interior for every $x \in B$.
Fact~\ref{fact:dimdsig} therefore implies that $B$ is $\DSig$. By the inductive hypothesis $B$ has nonempty interior.
It follows from Fact~\ref{KU} that $A$ has nonempty interior.
\end{proof}

Theorem D now follows immediately from SBCT and Proposition \ref{lmeasure}.

\section{Dimension Coincidence}
In this section $\mathcal R$ is an expansion of $(\R,<,+)$ that does not define a dense $\omega$-orderable set.
We first show that naive and topological dimension agree on $\DSig$ sets.
This is a consequence of Proposition~\ref{prop:raisedim} which states that a continuous definable map between $\DSig$ sets cannot raise topological dimension.
We first prove a definable selection theorem for $\DSig$ sets, Proposition~\ref{select} below.
Let $X \subseteq \R^n$, $f: X \to \R^m$, and $\epsilon > 0$.
We say that $f$ has \textbf{$\epsilon$-oscillation} at $x \in X$ if for all $\delta > 0$ there are $y,y' \in X$ such that $\| x - y \|,  \| x - y' \| < \delta$ and $\| f(y) - f(y') \| \geq \epsilon$.
The following fact is an easy analysis exercise.

\begin{fact}\label{basic metric}
For each $\epsilon > 0$ the set of $x \in X$ at which $f$ has $\epsilon$-oscillation is closed in $X$.
Furthermore, $f$ is discontinuous at $x \in X$ if and only if $f$ has $\epsilon$-oscillation at $x$ for some $\epsilon > 0$.
\end{fact}

\begin{lem}\label{oscil}
Let $U \subseteq \R^m$ be open and definable and $f: U \to \R^n$ be definable.
Then the set of points at which $f$ is discontinuous is $\DSig$.
Furthermore one of the following holds:
\begin{enumerate}
\item There is a definable open $V \subseteq U$ such that the restriction of $f$ to $V$ is continuous.
\item There is a definable open $V \subseteq U$ and an $\epsilon > 0$ such that $f$ has $\epsilon$-oscillation at every $p \in V$.
\end{enumerate}
\end{lem}

\begin{proof}
For each $\epsilon > 0$ let $A_\epsilon \subseteq U$ be the set of points at which $f$ has $\epsilon$-oscillation.
Let $A_0$ be the union of the $A_\epsilon$.
By Fact~\ref{basic metric}, $A_0$ is the set of points at which $f$ is discontinuous, and each $A_\epsilon$ is closed in $U$. Thus each $A_\epsilon$ is an intersection of an open set and a closed set.
It follows from Fact~\ref{basicDsigma}(3) that each $A_\epsilon$ is $\DSig$.
If $0 < s < t$ then $A_t \subseteq A_s$. Hence $A_0$ is $\DSig$ by Fact~\ref{basicDsigma}(4).
Applying SBCT we see that $A_0$ is either nowhere dense or $A_\epsilon$ has interior for some $\epsilon > 0$.
If $A_0$ is nowhere dense, then $(1)$ holds.
If $A_\epsilon$ has interior, then $(2)$ holds for this $\epsilon$.
\end{proof}

Let $U \subseteq \R^n$ be open and $f : U \to \R$.
We say $f$ is \textbf{Baire class one} if it is a pointwise limit of a sequence of continuous functions $g : U \to \R$. We say $f$ is \textbf{lower semi-continuous} if for any $p \in U$ and sequence $\{q_i\}_{i = 0}^{\infty}$ of elements of $U$ converging to $p$ we have
$$ \liminf_{i \to \infty} f(q_i) \geq f(p). $$
It is well-known, and not difficult to show, that a lower semi-continuous function is Baire class one.
It is also well-known that a Baire class one function $f : U \to \R$ is continuous on a comeager subset of $U$, see~\cite{kechris}.
In particular, if $f$ is Baire class one, then the set of points at which $f$ is continuous is dense in $U$.
Lemma~\ref{oscil} therefore implies the following.

\begin{cor}\label{cor:baireclassone}
Let $U \subseteq \R^n$ be open and definable and let $f : U \to \R$ be definable.
If $f$ is Baire class one, then $f$ is continuous on an open dense subset of $U$.
In particular, if $f$ is lower semi-continuous, then $f$ is continuous on an open dense subset of $U$.
\end{cor}

\noindent We recall the following elementary fact from real analysis. We leave the proof to the reader.

\begin{fact}\label{lem:closedsemicontinuous}
Let $A \subseteq \R^{m + 1}$ be compact and $U\subseteq \R^m$ be a nonempty, open and contained in the projection of $A$ onto the first $m$ coordinates.
Then the function $f: U \to \R$ that maps $p \in U$ to the minimal element of $A_p \subseteq \R$ is lower semi-continuous.
\end{fact}

\begin{prop}[$\DSig$-choice]\label{select}
Let $A \subseteq \R^{m + n}$ be $\DSig$ and let $\pi : \R^{m + n} \to \R^m$ be the projection onto the first $m$ coordinates.
If $\pi(A)$ has interior, then there is a definable open set $V \subseteq \pi(A)$ and a definable continuous function $f: V \to \R^n$ such that $\gr(f)\subseteq A$.
\end{prop}

\begin{proof}
Suppose that $\pi(A)$ has interior.
We first reduce to the case when $A$ is compact.
Let $\{ A_{s,t} : s,t > 0 \}$ be a definable family of compact subsets of $\R^{m + n}$ that witnesses that $A$ is $\DSig$.
Consider the definable family $\{ \pi(A_{s,t}) : s,t > 0\}$ of compact subsets of $\R^m$.
The union of the $\pi(A_{s,t})$ is equal to $\pi(A)$ and thus has interior.
By the Baire Category Theorem there are $s,t > 0$ such that $\pi(A_{s,t})$ has interior.
Fix such $s,t$.
It now suffices to show that there is a definable open $V \subseteq \pi(A_{s,t})$ and a definable continuous $f: V \to \R^n$ such that $\gr(f) \subseteq A_{s,t}$.
After replacing $A$ with this $A_{s,t}$, if necessary, we may assume that $A$ is compact.
Let $<_{lex}$ be the lexicographic order on $\R^n$.
Let $U$ be a definable open subset of $\pi(A)$.
If $p \in U$, then $A_p \subseteq \R^n$ is compact and nonempty and thus has a $<_{lex}$-minimal element.
Let $f: U \to \R^n$ be the definable function that maps $p \in U$ to the $<_{lex}$-minimal element of $A_p$.
We show, by applying induction to $n$, that $f$ is continuous on a nonempty open subset of $U$.
If $n = 1$, then $f$ is lower semi-continuous by Fact~\ref{lem:closedsemicontinuous}. Therefore $f$ is continuous on a nonempty open subset of $U$ by Corollary~\ref{cor:baireclassone}.
Suppose $n \geq 2$.
Let $\rho : \R^{m + n} \to \R^{m + n - 1}$ be the coordinate projection missing the last coordinate.
Let $B = \rho(A)$. Note that $B$ is compact.
Let $g : U \to \R^{n - 1}$ be the function that maps $p$ to the $<_{lex}$-minimal element of $B_p \subseteq \R^{n - 1}$.
Note that $g(p) = \rho(f(p))$ for all $p \in U$.
It follows from the inductive hypothesis that $g$ is continuous on a nonempty open subset of $U$.
After replacing $U$ with a smaller nonempty open set, we may suppose that $g$ is continuous on the closure of $U$.
Let $C \subseteq \cl(U) \times \R$ be the set of $(p,t)$ such that $(p,g(p),t) \in A$.
It is easy to see that $C$ is compact.
Let $h : U \to \R$ be given by declaring $h(p)$ to be the minimal $t \in \R$ such that $(p,t) \in C$.
It follows from the base case that $h$ is continuous on a nonempty open $V \subseteq U$.
Since $f(p) = (g(p),h(p))$ for every $p \in U$, we get that $f$ is continuous on $V$.
\end{proof}

\noindent Together with N\"obeling's projection theorem for topological dimension (see Fact \ref{topdimmono}(4)) this yields the following.

\begin{prop}\label{prop:raisedim}
Let $A \subseteq \R^n$ be $\DSig$ and let $f : \R^n \to \R^m$ be continuous and definable.
Then $\dim_T f(A) \leq \dim_T(A)$.
\end{prop}
\begin{proof}
Let $\dim_T f(A) = d$.
By Fact~\ref{topdimmono} there is a coordinate projection $\pi : \R^m \to \R^d$ such that the image of $f(A)$ under $\pi$ has interior.
After replacing $f$ with $\pi \circ f$, we may assume that $d = m$ and that $f(A)$ has interior.
Let $G \subseteq \R^{n + m}$ be the graph of the restriction of $f$ to $A$.
Then $G$ is the image of $A$ under the continuous definable map $p \mapsto(p,f(p))$ and is therefore $\DSig$.
Applying Proposition~\ref{select} to $G$ we obtain a nonempty definable open $V \subseteq f(A)$ and a continuous definable $h : V \to \R^n$ such that $(h(p),p) \in G$ for all $p \in V$.
Note that $f(h(p)) = p$ for all $p \in V$.
As $h$ is continuous, $\gr(h)$ is homeomorphic to $V$.
Therefore $\dim_T \gr(h) = m$.
Fact~\ref{topdimmono} yields $\dim_T(G) \geq m$.
Because $f$ is continuous, $G$ is homeomorphic to $A$. Thus $\dim_T(A) \geq m$.
\end{proof}

\noindent It follows immediately that if $f : \R^m \to \R^n$ is a continuous surjection and $m < n$, then $(\R,<,+,f)$ defines a dense $\omega$-orderable set.
 We now show that naive and topological dimension agree on $\DSig$ sets.

\begin{prop}\label{Bound0}
Let $A \subseteq \mathbb{R}^n$ be $\DSig$.
Then $\dim_T(A) = \dim A$.
\end{prop}

\begin{proof}
As $A$ is $F_\sigma$ it follows from Fact~\ref{topdimmono} that $\dim_T A \leq \dim A$.
Let $\pi : \R^n \to \R^{\dim A}$ be a coordinate projection such that $\pi(A)$ has interior.
Then $\dim_T \pi(A) = \dim A$. By Proposition~\ref{prop:raisedim}, $\dim_T A \geq \dim A$.
\end{proof}

\subsection{Proof of Theorem A}
Recall our standing assumption that $\Cal R$ does not admit a dense $\omega$-orderable set.
Theorem~\ref{hausdorff} below shows that if $A \subseteq \R^n$ is closed and satisfies $\dim_T(A) < \dim_H(A)$, then $( \mathcal{M}_\R, A)$ defines a dense $\omega$-orderable set, and hence defines all compact sets by Theorem C.
Theorem A follows.

\begin{theorem}\label{hausdorff}
Suppose that $\Cal R$ expands $\mathcal{M}_\R$.
Let $A\subseteq \R^n$ be $\DSig$. Then $\dim_T(A) = \dim_H(A)$.
\end{theorem}

\begin{proof}
By Fact~\ref{topdimmono} we have $\dim_T(A) \leq \dim_H(A)$.
We establish $\dim_H(A) \leq \dim_T(A)$ by proving the following inequality for all $k \geq 2$:
 \begin{equation}\label{eq:dimco1}
     \dim_H(A) \leq \dim_T(A) + \frac{1}{k}.\tag{*}
 \end{equation}
Fix $k \geq 2$.
We first show that
\begin{equation}\label{eq:dimco2}
\dim_H(A^k) \leq \dim_T(A^k) + 1. \tag{**}
\end{equation}
To prove \eqref{eq:dimco2} it is enough to show that whenever $\dim_H(A^k) > m$, then $\dim_T(A^k) \geq m$. Suppose that $\dim_H(A^k)>m$.
By the Marstrand projection theorem (Fact~\ref{marstrand}) there is an orthogonal projection $\rho : \mathbb{R}^{kn} \to \mathbb{R}^m$ such that $\rho(A^k)$ has positive $m$-dimensional Lebesgue measure. Since $\Cal R$ expands $\mathcal M_{\R}$, $\rho$ is definable. Because $\rho(A^k)$ is $D_\Sigma$, it follows from Proposition~\ref{lmeasure} that $\rho(A^k)$ has nonempty interior.
Thus $\dim_T \rho(A^k) = m$.
As $\rho$ is definable Proposition~\ref{prop:raisedim} yields $\dim_T(A^k) \geq m$.
Thus \eqref{eq:dimco2} holds. \newline
We now show \eqref{eq:dimco1}. First note that $k \dim_H(A) \leq \dim_H(A^k)$ by Fact~\ref{Hsadditive}. Moreover,
Fact~\ref{fact:naivedim} and Proposition~\ref{Bound0} imply $\dim_T(A^k) = k \dim_T(A)$. Thus we get from \eqref{eq:dimco2} that
$$ k \dim_H(A) \leq k \dim_T(A) + 1. $$
Dividing by $k$ yields the desired inequality \eqref{eq:dimco1}.
\end{proof}

\noindent As observed in the introduction, Theorem \ref{hausdorff} fails when `\emph{expands $\mathcal{M}_{\R}$}' is replaced by `\emph{expands $\mathcal{M}_K$ for some countable subfield $K$ of $\R$}'. We do not know whether the same is true when $K$ ranges over uncountable subfields of $\R$. We leave this as an open question.

\begin{ques} Is there an uncountable subfield $K\subseteq \R$ and a closed subset $A\subseteq \R^n$ such that $\dim_T A < \dim_H A$ and $(\mathcal M_K,A)$ avoids a compact set?
\end{ques}

\section{Properties of topological dimension}\label{naive}

Throughout this section let $\Cal R$ be an expansion of $(\R,<,+)$ that does not define a dense $\omega$-orderable set. \emph{`Definable'} will mean \emph{`definable in $\Cal R$ possibly with parameters'}. By Proposition~\ref{Bound0} the results of this section hold with topological dimension in place of naive dimension. Theorem E follows.

\begin{cor}\label{dimclosure}
Let $A \subseteq \R^n$ be $\DSig$.
Then $\dim \cl(A) = \dim A$ and $\dim \bd(A) < n$.
\end{cor}

\begin{proof}
By monotonicity of the naive dimension, $\dim \cl(A) \geq \dim A$. We first show that $\dim A \geq \dim \cl(A)$.
Let $m = \dim \cl(A)$.
Let $\pi : \R^n \to \R^m$ be a coordinate projection such that the image of $\cl(A)$ under $\pi$ contains an open set $U$.
Because $A$ is dense in $\cl(A)$, we see that $\pi(A)$ is dense in $U$.
By SBCT, $\pi(A)$ has interior in $U$. Thus $\dim(A) \geq m$.
The second claim follows from SBCT, because the boundary of a subset of $\R^n$ has interior if and only if the set is dense and co-dense in some nonempty open set.
\end{proof}

\begin{cor}\label{cor:dimension} Let $A,B \subseteq \R^n$ be $\DSig$. Then $\dim A \cup B = \max \{ \dim A, \dim B\}$.
\end{cor}
\begin{proof} By Fact \ref{fact:naivedim} it is enough to show that $\dim (A \cup B) \leq \max \{ \dim A, \dim B\}$.
Let $d = \dim (A \cup B)$. Let $\pi :
\R^n \to \R^d$ be a coordinate projection such that $\pi(A \cup B)$ has interior. We just need to show that either $\pi(A)$ or $\pi(B)$ has interior.
If neither $\pi(A)$ nor $\pi(B)$ has interior, then by SBCT both $\pi(A)$ and $\pi(B)$ are nowhere dense. As a finite union of nowhere dense sets is nowhere dense, $\pi(A\cup B)$ is nowhere dense. This contradicts our assumption that $\pi(A\cup B)$ has interior.
\end{proof}

\begin{lem}\label{lem:complinDSig} Let $A,A' \subseteq \R^n$ be $D_{\Sigma}$ such that $\dim A' < \dim A$. Then there is a $D_{\Sigma}$ set $B \subseteq A \setminus A'$ with $\dim B = \dim A$.
\end{lem}
\begin{proof}
By Corollary~\ref{dimclosure} we have $\dim \cl(A') < \dim A$.
Note that $B := A \setminus \cl(A')$ is $\DSig$. By Corollary~\ref{cor:dimension}, $\dim A = \max \{ \dim \cl(A'), \dim B \}$. Thus $\dim A = \dim B$.
\end{proof}

\subsection{Dimension additivity} In this section we will show the following additivity theorem for naive dimension.
\begin{thm}\label{thm:additivity} Let $A\subseteq \R^m \times \R^n$ be $D_{\Sigma}$ such that $\dim A=d$.  Then there is $e\leq d$ such that
\[
\dim \{ x \in \pi(A) \ : \ \dim A_x \geq e \} \geq d - e,
\]
where $\pi : \R^{m+n} \to \R^m$ is the projection onto the first $m$-coordinates.
\end{thm}
\noindent To prove Theorem \ref{thm:additivity} we need to keep track of various coordinate projections. We fix the following notation. Let $[n]$ be the set $\{1,\dots,n\}$. Let $d\in \N, i_1,\dots,i_d \in [n]$ and $J\subseteq [n]$ such that $i_1 <  \ldots < i_d$ and $J=\{i_1,\dots,i_d\}$. We let $\pi_J$ be the coordinate projection $\R^n \to \R^d$ that projects onto the coordinates $i_1,\dots,i_d$.

\begin{lem}\label{lem:easypart} Let $A\subseteq \R^m \times \R^n$ be $D_{\Sigma}$ and $d,e \in \N$ be such that
\[
\dim \{ x \in \pi_{[m]}(A) \ : \ \dim A_x \geq e \} \geq d - e.
\]
Then $\dim A \geq d$.
\end{lem}
\begin{proof} For $J \subseteq [n]$ with $|J|=e$ define
$B_J := \{ x \in \pi_{[m]}(A) \ : \ \pi_J(A_x) \hbox{ has interior} \}.$
By Fact \ref{fact:dimdsig} $B_J$ is $\DSig$ for every such $J$. By our assumption on $A$,
\[
\dim \left(\bigcup_{J\subseteq [n], |J|=e} B_J\right) \geq d-e.
\]
Therefore by Corollary \ref{cor:dimension} we can pick a $J_0\subseteq [n]$ such that $|J_0|=e$ and $\dim B_{J_0} \geq d-e$. Now consider
\[
A' := \{ (x,y) \in A \ : \ x \in B_{J_0} \}.
\]
Let $I\subseteq [m]$ be such that $|I|=d-e$ and $\pi_I(B_{J_0})$ has interior. Then $\pi_{I\cup J_0}(A')$ has interior by Fact \ref{KU}. Thus $\dim A \geq \dim A' \geq d$.
\end{proof}

\begin{lem}\label{lem:interior} Let $A\subseteq \R^{1+m}$ be $\DSig$ such that $\dim A=d$ and $\dim A_x< d$ for all $x \in \pi_{[1]}(A)$. Then $\pi_{[1]}(A)$ has interior.
\end{lem}
\begin{proof}
We suppose towards a contradiction that $\pi_{[1]}(A)$ has empty interior. As $\dim A=d$, we can assume without loss of generality that $\pi_J(A)$ has interior for $J=\{2,\dots,d+1\}$.
By Proposition \ref{select} there is a connected open set $U\subseteq \pi_{J}(A)$ and a continuous definable map $f: U \to \R$ such that
\[
\{ (f(x),x) \ : \ x \in U \} \subseteq \pi_{[d+1]}(A).
\]
By the intermediate value theorem, $f(U)$ is connected. Since $\pi_{[1]}(A)$ is totally disconnected, $f$ is necessarily constant. Thus there is an $x \in \pi_{[1]}(A)$ such that $\pi_{[d+1]}(A)_x$ contains $U$. Thus for this $x$ we have $\dim A_x \geq d$, contradicting our assumptions on $A$.
\end{proof}

\begin{proof}[Proof of Theorem \ref{thm:additivity}]
We proceed by induction on the triple $(m+n,m,d)\in \N^{3}$, where $\N^3$ is ordered lexicographically.  The result is immediate when $m=0$ or $n=0$ or $d=0$.
Thus we may assume that $m>0,n>0$ and $d>0$, and that the desired result holds for all triples lexicographically smaller than $(m+n,m,d)$. \newline

\noindent We first consider the case $m=1$. In this case we need to show that either
\begin{itemize}
\item[(i)] there is an $x\in \pi_{[1]}(A)$ such that $\dim A_x = d$, or
\item[(ii)] $\{ x\in \pi_{[1]}(A) \ : \ \dim(A_x) \geq d-1\}$ has interior.
\end{itemize}
Let $J\subseteq [1+n]$ be such that $|J|=d$ and $\pi_J(A)$ has interior. If $d<n$, then by applying the induction hypothesis to $\pi_{\{1\} \cup J}(A)$, we have one of the following
\begin{itemize}
\item[(a)] either is $x\in \pi_{[1]}(A)$ such that $\dim \pi_{\{1\} \cup J}(A)_x = d$ or
\item[(b)] $\{ x\in \pi_{[1]}(A) \ : \ \dim(\pi_{\{1\} \cup J}(A)_x) \geq d-1\}$ has interior.
\end{itemize}
Observe that (a) implies (i), while (b) implies (ii). \newline

\noindent We have reduced the case $m=1$ to the case when $m=1$ and $d=n$. We suppose towards a contradiction that both (i) and (ii) fail. Since (i) fails, $\dim A_x< n$ for all $x\in \pi_{[1]}(A)$. Applying Lemma \ref{lem:interior} to the set
\[
A' :=\{ (x,u) \in (\R \times \R^n)\cap A \ : \ \dim(A_x) \geq n-1\},
\] we can deduce from the failure of (i) and (ii) that $\dim A'<n$. Thus by Lemma \ref{lem:complinDSig} there is a $\DSig$ subset $B$ of $A$ such that $\dim B = n$ and $B \subseteq A\setminus A'$. By definition of $A'$ we have $\dim B_x < n-1$ for all $x \in \pi_{[1]}(B)$. We assume without loss of generality that $A=B$. Thus $\dim(A_x) < n-1$  for all $x\in \pi_{[1]}(A)$. Set
\[
A'':=\{ u \in \pi_{[n]}(A)\ : \ \dim A_{u} \geq 1\}.
\]
We now show that $\dim A''<n-1$. Observe that for $x \in \pi_{[1]}(A)$
\[
A''_x = \{ v \in \pi_{[n-1]}(A_x) \ : \ \dim A_{(x,v)} \geq 1\}.
\]
Applying Lemma~\ref{lem:easypart} to $A_x$, we see that $\dim  A''_x  < n-2$
for all $x \in \pi_{[1]}(A)$. Applying the inductive hypothesis to $A''$ yields $\dim A'' < n - 1.$
Thus
\[
\dim \{ z\in A \ : \ \pi_{[n]}(z) \in A''\} < n.
\]
By Lemma \ref{lem:complinDSig} there is a $D_\Sigma$ subset $C$  of $A$ such that $\dim C=n$ and $\dim A_u =0$ for every $u \in \pi_{[n]}(C)$. We suppose without loss of generality that $A = C$. By Proposition \ref{select} there is an open $U\subseteq \R^n$ and a continuous definable $f : U \to \R^{1}$ such that $\gr(f) \subseteq A$. We assume without loss of generality that $A = \gr(f)$. Let $(y,z) \in (\R^{n-1} \times \R) \cap U$. By applying Fact \ref{fact:monotonicity} to $f(y,-)$ we get an open dense subset $V$ of $U_y$ such that the restriction of $f(y,-)$ to each connected component of $V$ is either constant or strictly monotone. As $\dim A_{x,y}=0$ for every $x \in \pi_{[1]}(A)$, the restriction of $f(y,-)$ to any connected component of $V$ is not constant. Thus for every $y \in \pi_{[n-1]}(U)$ there is an open interval $I_y$ such that $f(y,-)$ is strictly monotone on $I_y$. By the Intermediate Value Theorem, the set
\[
\{ x \in \pi_{[1]}(A) \ : \ \exists z \in I_y \ f(y,z) = x\}
\]
has interior for each $y\in \pi_{[n-1]}(U)$. Thus the set $\{ x \in \pi_{[1]}(A) \ : \ (x,y) \in \pi_{[n]}(A)\}$ has interior for every $y \in \pi_{[n-1]}(U)$. Therefore $\dim \pi_{[n]}(A) = n$ by Lemma \ref{lem:easypart}. This contradicts the failure of (ii).\newline

\noindent We finally consider the case $m>1$. Let $A \subseteq \R^{m+n}$ be $\DSig$ such that $\dim A = d$. As $m>1$ the induction hypothesis applied to $A$ (considered as a subset of $\R \times \R^{m+n-1}$)
shows that either
\begin{itemize}
\item[(i')] there is $x\in \pi_{[1]}(A)$ such that $\dim A_x = d$, or
\item[(ii')] $\{ x\in \pi_{[1]}(A) \ : \ \dim(A_x) \geq d-1\}$ has interior.
\end{itemize}
Suppose (i') holds and let $x\in \pi_{[1]}(A)$ be such that $\dim A_x = d$. By applying the inductive hypothesis to $A_x\subseteq \R^{m-1}\times \R^{n}$, there is an $e\leq d$ such that
\[
\dim \{ y \in \pi_{[m-1]}(A_x) \ : \ \dim A_{x,y} \geq e \} \geq d-e.
\]
Thus $\dim \{ (x,y) \in \pi_{[m]}(A) \ : \ \dim A_{x,y} \geq e \} \geq d-e$ and we are done in this case. Now suppose that (ii') holds.
For each $e\leq d-1$ set
\[
B_e := \{ x \in \pi_{[1]}(A) \ : \ \{ y \in \pi_{[m-1]}(A_x) \ : \ \dim A_{x,y} \geq e \} \geq d-1-e\}.
\]
By (ii') and the induction hypothesis (applied to $A_x\subseteq \R^m\times \R^{n-1}$), the set $\bigcup_{e\leq d-1} B_e$ has interior. Since each $B_e$ is $\DSig$, there is an $e_0\leq d-1$ such that $B_{e_0}$ has interior.
Consider $C := A \cap (B_{e_0} \times \R^{m-1+n})$. By Lemma \ref{lem:easypart} and the fact that $B_{e_0}$ has interior we have
\[
\dim \{ (x,y) \in \pi_{[m]}(A') \ : \ \dim A_{x,y} \geq e_0 \} \geq d-e_0.
\]
\end{proof}

\begin{corollary}\label{raisedim}
Let $A \subseteq \R^m$ and $B\subseteq \R^n$ be $D_\Sigma$ and $f : \R^m \to \R^n$ be continuous and definable.
If there is a $d\in \N$ such that $\dim f^{-1}(x) = d$ for all $x\in B$, then
\[
\dim f^{-1}(B) = \dim B + d.
\]
\end{corollary}

\begin{proof}
Let $G \subseteq \R^{m+n}$ be the graph of $f$. Then $f^{-1}(B)=\pi_{[m]}\left(G \cap (\R^m \times B)\right)$ and is thus $\DSig$. Applying Theorem \ref{thm:additivity} to $f^{-1}(B)$ yields $f^{-1}(B) \leq \dim B + d$. The inverse inequality follows from Lemma \ref{lem:easypart}.
\end{proof}

\section{Totally disconnected sets}

In this section we give a proof of Theorem B and show that Theorem A fails when Hausdorff dimension is replaced by packing dimension.

\subsection{Proof of Theorem B} We first recall a few results and definitions. Let $\Cal R$ be an expansion of $(\R,<,+)$. Then $\Cal R$ is \textbf{o-minimal} if every $\mathcal{R}$-definable subset of $\R$ is a finite union of open intervals and singletons. It is well-known that $\mathcal{M}_{\R}$ is o-minimal, see van den Dries \cite[Chapter 1.7]{Lou}. Given $E \subseteq \R^n$, we denote by $(\mathcal{R},E)^\sharp$ the expansion of $\mathcal{R}$ by all subsets of all cartesian powers of $E$.

\begin{fact}[{\cite[Theorem A]{FM-sparse}}]\label{HF}
Let $\mathcal{R}$ be an o-minimal expansion of $(\R,<,+)$. Let $E \subseteq \R$ be such that $f(E^k)$ is nowhere dense for every $\mathcal{R}$-definable $f: \R^k \to \R$.
Then every subset of $\R$ definable in $(\mathcal{R},E)^\sharp$ either has interior or is nowhere dense.
\end{fact}

The following corollary is a special case of Fact~\ref{HF} as the semilinear cell decomposition \cite[Corollary 1.7.8]{Lou} implies that every $\mathcal{M}_\R$-definable function $f : \R^k \to \R$ is piecewise affine.
We leave the details to the reader.

\begin{cor}\label{linearHF}
Let $E \subseteq \R$ be such that $T(E^k)$ is nowhere dense for every linear $T : \R^k \to \R$.
Then every subset of $\R$ definable in $(\mathcal{M}_{\R},E)^\sharp$ either has interior or is nowhere dense.
\end{cor}


The next theorem characterizes compact totally disconnected $E \subseteq \R^n$ for which $(\mathcal{M}_{\R},E)$ avoids a compact set.

\begin{thm}\label{seven}
Let $E \subseteq \R^n$ be compact and totally disconnected.
Then the following are equivalent:
\begin{itemize}
\item [(i)] $(\mathcal{M}_{\R},E)$ avoids a compact set.
\item [(ii)] $(\mathcal{M}_{\R},E)$ does not define a dense $\omega$-orderable set.
\item [(iii)] $T(E^k)$ is nowhere dense for all $k$ and all linear maps $T : \R^{kn} \to \R$ .
\item [(iv)] $E^k$ has Hausdorff dimension zero for all $k$.
\item [(v)] Every $(\mathcal{M}_{\R},E)^\sharp$-definable subset of $\R$ has interior or is nowhere dense.
\item [(vi)] Every $(\mathcal{M}_{\R},E)$-definable subset of $\R$ has interior or is nowhere dense.
\end{itemize}
\end{thm}

We first prove a lemma.

\begin{lem}\label{up}
Let $E \subseteq \R^n$ and let $A \subseteq \R$ be the set of coordinates of elements of $E$.
If $\dim_H(E^k) = 0$ for all $k$, then $\dim_H(A^k) = 0$ for all $k$.
\end{lem}

\begin{proof}
Suppose that $\dim_H(E^k) = 0$ for all $k$.
Fix $k$.
For $1 \leq i \leq n$ let $\pi_i : \R^n \to \R$ be the projection onto the $i$th coordinate.
For each function $\sigma : \{ 1,\ldots,k \} \to \{ 1, \ldots, n \}$ we let $\rho_{\sigma} : E^k \to \R^{k}$ be given by
$$ \rho_\sigma(x_1,\ldots,x_k) = (\pi_{\sigma(1)}(x_1),\ldots,\pi_{\sigma(k)}(x_k)) \quad \text{for all } x_1,\ldots,x_k \in E. $$
Then $A^k$ is the union of the $\rho_\sigma(E^k)$.
By assumption $\dim_H (E^k) = 0$ for all $\sigma$. Therefore by Fact~\ref{Hsadditive} we have $\dim_H \rho_\sigma(E^k) = 0$ for all $\sigma
$.
Thus $\dim_H(A^k) = 0$.
\end{proof}

\begin{proof}[Proof of Theorem~\ref{seven}]
Theorem C shows that (i) and (ii) are equivalent. We show that (ii) implies (iii).
Suppose that $(\mathcal{M}_{\R},E)$ does not define a dense $\omega$-orderable set.
Every finite cartesian power of a totally disconnected set is totally disconnected. Thus $\dim_T(E^k) = 0$ for all $k$.
Let $T : \R^{kn} \to \R$ be linear.
As $\dim_T(E^k) = 0$ and $T$ is definable, it follows from Proposition~\ref{prop:raisedim} that $\dim_T T(E^k) = 0$.
Thus $T(E^k)$ has empty interior.
Because $T(E^k)$ is compact, it follows that $T(E^k)$ is nowhere dense.\newline
We show that (iii) implies (iv) by contrapositive.
Suppose $\dim_H(E^k) > 0$.
Let $m\in \N$ be such that $ \dim_H(E^k) > \frac{1}{m}$.
By Fact~\ref{Hsadditive}
$$ \dim_H(E^{mk}) \geq m \dim_H(E^k) > 1. $$
By the Marstrand projection theorem (Fact~\ref{marstrand}) there is an orthogonal projection $\rho : \R^{mk} \to \R$ such that $\rho(E^k)$ has positive Lebesgue measure.
As $\rho(E^k)$ has positive Lebesugue measure, it follows from a classical theorem of Steinhaus~\cite{Stein} that the set
$$ \{ t - t' : t,t' \in \rho(E^k)\} $$
has nonempty interior in $\R$.
Let $T : \R^{mk + mk} \to \R$ be given by
$$ T(x,x') = \rho(x) - \rho(x') \quad \text{for all } x,x' \in \R^{mk}. $$
Then $T$ is linear and $T(E^{2k})$ has interior.\newline
We show that (iv) implies (v).
Suppose that $\dim_H(E^k) = 0$ for all $k$.
We let $A \subseteq \R$ be the set of coordinates of elements of $E$. Observe that $A$ is compact because $E$ is compact. Moreover, $(\mathcal{M}_{\R},E)^\sharp$ is a reduct of $(\mathcal{M}_{\R},A)^\sharp$ as $E \subseteq A^n$.
We show that every $(\mathcal{M}_{\R},A)^\sharp$-definable subset of $\R$ has interior or is nowhere dense.
By Lemma~\ref{up}, $A^k$ has Hausdorff dimension zero for all $k$.
Let  $T : \R^k \to \R$ be linear. Because $\dim_H(A^k) = 0$, it follows from Fact~\ref{Hsadditive} that $\dim_H T(A^k) = 0$.
Thus $T(A^k)$ has empty interior. As $T(A^k)$ is compact, it follows that $T(A^k)$ is nowhere dense. Now apply Corollary~\ref{linearHF}.\newline
It is clear that (v) implies (vi). Statement (ii) follows from (vi), because a dense $\omega$-orderable set cannot have interior.
\end{proof}

\begin{corollary}
Let $A\subseteq \R$ be bounded such that for every open interval $I\subseteq \R$ the set $I\cap A$ is not dense and co-dense in $I$.
Then $(\mathcal{M}_\R, A)$ avoids a compact set if and only if $\bd(A)^k$ has Hausdorff dimension zero for every $k$.
\end{corollary}

\begin{proof}
Since $A$ is bounded, $\bd(A)$ is compact.
Because $A$ is not dense and co-dense in any open interval, $\bd(A)$ is nowhere dense.
Thus, if $(\mathcal{M}_\R,A)$ avoids a compact set, then $\bd(A)^k$ has Hausdorff dimension zero for all $k$ by Theorem~\ref{seven}.
\newline

For the proof of the other implication, let $E = \bd(A)$ and assume that $E^k$ has Hausdorff dimension zero for all $k$.
Since $E$ is compact and nowhere dense, it follows from Theorem~\ref{seven} that $(\mathcal{M}_\R,E)^\sharp$ avoids a compact set.
We show that $A$ is $(\mathcal{M}_\R,E)^\sharp$-definable.
Note that $A \setminus \Int(A)$ is a subset of $E$ and therefore is $(\mathcal{M}_\R,E)^\sharp$-definable.
We show that $\Int(A)$ is $(\mathcal{M}_\R,E)^\sharp$-definable.
Without loss of generality we may assume that $A = \Int(A)$.
Then  $A$ decomposes as the union of its connected components, which are bounded open intervals.
Let $B$ be the set of $(x,y) \in E^2$ such that $x < y$ and $\{ t \in \R : x < t < y\}$ is a connected component of $A$. Note that $B$ is definable in $(\mathcal{M}_\R,E)^\sharp$.
Then
$$ A = \bigcup_{ (x,y) \in B } \{ t \in \R : x < t < y \}.$$
Thus $A$ is $(\mathcal{M}_\R,E)^\sharp$-definable because $B$ is.
\end{proof}

\subsection{Optimality of Theorem A}
Theorem~\ref{seven} may be used to show that Theorem~\ref{hausdorff} does not hold with packing dimension in place of Hausdorff dimension.
We construct a compact $A \subseteq \R$ such that $(\mathcal{M}_{\R},A)$ avoids a compact set and $0 < \dim_P(A) < 1$.
Our $A$ will be nowhere dense because $A$ is compact and $\dim_P(A) < 1$.
By Theorem~\ref{seven} it suffices to construct a compact $A \subseteq \R$ such that $\dim_H(A^k) = 0$ for all $k$ and $0 < \dim_P(A) < 1$.
Given a subset $S$ of the positive integers we let
$$  \underline{d}(S) = \liminf_{ n \to \infty } \frac{ | S \cap \{ 1 , \ldots, n \} | }{n}
\quad \text{and} \quad
\overline{d}(S) = \limsup_{ n \to \infty } \frac{ | S \cap \{ 1 , \ldots, n \}| }{n} \quad $$
be the lower and upper densities of $S$, respectively.
Let $E_S$ be the set of $t \in [0,1]$ such that the $k$th digit of the binary expansion of $t$ is zero whenever $k\notin S$.
Note that $E_S$ is compact for every $S\subseteq \N$. The following is a special case of Lemma 1 of \cite{WWW}.

\begin{fact}\label{www}
Let $S$ be a set of positive integers and $k\in \N$.
Then
$ \dim_H(E_S^k) = k \underline{d}(S)$ and $\dim_P(E_S^k) = k \overline{d}(S)$.
\end{fact}

Define a sequence $(a_n)_{n \in \mathbb{N}}$ of natural numbers by declaring $a_0 = 2$ and $a_{n + 1} = 2^{a_n}$ for all $n$.
Let $S$ be the set of $m\in \N$ such that $a_n \leq m \leq 2a_n$ for some $n$. It is easy to see that $\underline{d}(S) = 0$ and $\overline{d}(S) = \frac{1}{2}$. By Fact \ref{www} $\dim_H E_S^k = 0$ for all $k$ and $\dim_P E_S =\frac12$. Thus $(\mathcal{M}_{\R},E_S)$ avoids a compact set but $0<\dim_P(E_S) <1$.

\bibliographystyle{alpha}
\bibliography{HW-Bib}
\end{document}